\RequirePackage{fix-cm}
\pdfoutput=1
\documentclass[smallextended]{svjour3}       
\smartqed  
\usepackage{graphicx}
\usepackage{amssymb}
\usepackage{amsmath}
\usepackage{color}
\usepackage{url} 
\usepackage{array}
\usepackage{booktabs}

\newcommand{\cvx}[1]{\mathrm{conv}\left( #1 \right)}
\newcommand{\tr}{\mathsf{T}}
\newcommand{\QED}{\hfill \ensuremath{\Box}}
\newcommand{\order}{\mathbf{O}}

\newcommand{\epsL}{\epsilon_{\mathrm{L}}}

\newtheorem{assumption}{Assumption}

\definecolor{boris}{rgb}{0.0, 0.4, 0.1}

\begin{document}

\title{Partially Distributed Outer Approximation
}

\author{ Alexander Murray \and Timm Faulwasser \and Veit Hagenmeyer \and Mario E.~Villanueva \and Boris Houska
}


\institute{A.~Murray, T.~Faulwasser, V.~Hagenmeyer \at
              Institute for Automation and Applied Computer Science, Karlsruhe Institute of Technology, Germany
              \email{\{alexander.murray, veit.hagenmeyer\}@kit.edu, timm.faulwasser@ieee.org}  
           \and
           M.E.~Villanueva and B.~Houska \at
           School of Information Science and Technology, ShanghaiTech University, China
           \email{\{meduardov,borish\}@shanghaitech.edu.cn}
}

\date{Received: date / Accepted: date}

\maketitle

\begin{abstract}
This paper presents a novel partially distributed outer approximation algorithm, named PaDOA, for solving a class of structured mixed integer convex programming (MICP) problems to global optimality. The proposed scheme uses an iterative outer approximation method for coupled mixed integer optimization problems with separable convex objective functions, affine coupling constraints, and compact domain. PaDOA proceeds by alternating between solving large-scale structured mixed-integer linear programming problems and partially decoupled mixed-integer nonlinear programming subproblems that comprise much fewer integer variables. We establish conditions under which PaDOA converges to global minimizers after a finite number of iterations and verify these properties with an application to thermostatically controlled loads.
\keywords{Mixed Integer Programming \and Distributed Optimization \and Outer Approximation \and Global Optimization}
\end{abstract}

\section{Introduction}
\label{sec::intro}

A mixed integer convex program (MICP) is an optimization problem with convex objective and constraint functions, where the only non-convex constraint is that a subset of the optimization variables need to be integer-valued~\cite{Bonami2012,Lubin2017}. MICPs arise in a plethora of application areas ranging from AC transmission expansion planninh and robust power flow problems~\cite{Alguacil2003,Kocuk2017}, via thermal unit design and control~\cite{Carrion2006}, a variety of scheduling and layout design problems~\cite{Sawaya2006}, design of multi-product batch plants~\cite{Ravemark1998}, to obstacle avoidance and robotic motion planning problems~\cite{Kuindersma2016}.

Although MICPs are NP-hard in general, there exist a variety of algorithms for solving MICPs to global optimality~\cite{Bonami2012}. 
State-of-the-art MICP solvers are based on tailored methods that exploit the fact that the integrality constraints are discrete while all other constraints are convex. Early attempts to develop tailored branch \& bound methods for MICP have been proposed in~\cite{Gupta1985}, mostly focussing on computational experiments and heuristics for selecting the branching variables and nodes. Improved versions of these early branch \& bound methods for MICP can be found in~\cite{Borchers1994,Leyffer2001}. Other early methods for solving MICP include generalized Benders decomposition methods~\cite{Benders_1962,Geoffrion1972}, which are, however, less frequently used in state-of-the-art MICP solvers.\footnote{For more details see~\cite{Bonami2012}.}

Modern MICP implementations are often, in one or the other way, based on or related to outer approximation (OA), which goes back to Duran and Grossmann~\cite{Duran_1986}. In contrast to branch \& bound, OA alternates between solving nonlinear programs (NLP) with fixed integer values as well as mixed integer linear programs (MILP), which are constructed by linearizing the objective and constraint functions at the solutions of the NLP and which are used to update the integer variables.
A notable extension of OA has been developed by Fletcher and Leyffer~\cite{Fletcher_1994}, who suggest to include curvature information in the relaxed integer program leading to a quadratic outer approximation method. Moreover, Kesavan and co-workers~\cite{Kesavan_2004} have studied variants of OA for solving non-convex mixed integer problems.
Another class of MICP methods are based on (extended) cutting plane methods~\cite{Westerlund1995} or combination of OA and branch-and-cut~\cite{Quesada1992}; see also~\cite{Tawarmalani2005} for a general overview of polyhedral branch-and-cut methods. In recent years, there has been considerable progress in lift-and-project methods for MICP. An excellent overview and discussion of the state-of-the-art of such lift-and-project methods can be found in a recent article by Kilin\c{c}, Linderoth, and Luedtke~\cite{Kilinc2017}.

Another recent trend in MICP solver development is the exploitation of separable structures by so-called extended formulations~\cite{Hijazi2014}. Here, the main idea is to introduce auxiliary variables in order to bound decoupled summands in additive expressions separately~\cite{Vielma2016}, which can lead to tighter polyhedral outer approximations. Such extended formulations have not only found their way into OA methods, as discussed in~\cite{Hijazi2014}, but they can also be used to increase the performance of lift-and-project methods for MICP~\cite{Kilinc2017}. However, extended formulations exploit the separability for the construction of tighter outer approximations only, but neither existing OA methods nor state-of-the-art lift-and-project methods ever attempt to break a large-scale MICP into decoupled MICPs with fewer integer variables. This is in contrast to distributed continuous convex optimization methods, such as dual decomposition~\cite{Everett1963,Necoara2008}, alternating direction method of multipliers (ADMM)~\cite{Boyd_2011,Eckstein1992,Gabay1976}, or augmented lagrangian based alternating direction inexact newton (ALADIN) methods~\cite{Houska_2016}, which can all be used to solve large-scale convex optimization problems to global optimality by alternating between solving small-scale convex optimization problems and sparse linear algebra operations. These methods typically require communication of the solutions of the decoupled problems between neighbors or to a central coordinator~\cite{Boyd_2011}. Although some researchers,~\cite{Takapoui_2016,Murray_2018}, have attempted to apply these distributed local optimization methods in a heuristic manner, these methods cannot find global minimizers of non-convex problems reliably. This is due to the fact that ADMM, ALADIN, or similar distributed convex optimization method typically rely on strong duality results for augmented Lagrangians~\cite{Shapiro2004,Shapiro2009}, which fail to hold in the presence of integrality constraints.

After reviewing Hijazi's extended formulations and related existing outer approximation methods in Section~\ref{sec::OA}, the main contribution of this paper is presented in Section~\ref{sec::PADOA}, which introduces
a partially distributed outer approximation (PaDOA) method for MICPs with separable objective functions. In contrast to existing algorithm for structured MICP, PaDOA alternates between solving MICPs with fewer integer variables and large-scale MILPs for which efficient algorithms exist. Section~\ref{sec::convergenceAnalysis} discusses the global convergence properties of PaDOA, as summarized in Theorem~\ref{thm::PaDOA}. In this context, we additionally establish the fact that global optimality of a given feasible point of an MICP with $N$ separable objectives and $Nn$ optimization variables can be computationally verified by solving $N$ partially-decoupled MICPs, each comprising at most $n$ local integer variables, and one MILP with $N n$ integer variables. This result is summarized in Theorem~\ref{thm::termination}, which analyzes one-step convergence conditions for PaDOA. In the sense that both MICPs as well as MILPs are NP hard in general~\cite{Garey1979,Murty1987}, this result is not in conflict with existing complexity results for mixed integer optimization problems. However, there are solvers such as \texttt{CPLEX}~\cite{CPLEX2009}, \texttt{Gurobi}~\cite{GUROBI2009}, and many others~\cite{Conforti2009}, which can solve practical MILPs within reasonable computational run-times. Thus, the fact that one can reduce the task of verifying global optimality of a feasible point of a separable MICP with coupled affine constraints to the task of solving one MILP of a comparable size and several smaller subproblems, is---at least from a computational perspective---an important contribution.
Last but not least, Section~\ref{sec:sim} illustrates the practical performance of PaDOA by applying the algorithm to MICP benchmark case studies. Section~\ref{sec::conclusions} concludes the paper.

\subsection{Problem formulation}
\label{sec:prob}
The present paper is concerned with mixed integer optimization problems of the form
\begin{eqnarray}
\label{eq::minlp}
\begin{array}{rccl}
V^\star &=& \underset{x \in X, z \in Z}{\min}& f(x,z) \\[0.3cm]
& & \text{s.t.}& A x = b
\end{array}
\end{eqnarray}
with separable objective function
$f(x,z) = \sum_{i=1}^N f_i(x_i,z_i)$
and separable constraint sets
\begin{eqnarray}
\begin{array}{rclcrcl}
X &=& X_1 \times \ldots \times X_N \quad &\text{with}& \quad X_1, X_2, \ldots, X_N &\subseteq& \mathbb R^{n} \\[0.16cm]
\text{and} \quad Z &=& Z_1 \times \ldots \times Z_N \quad &\text{with}& \quad Z_1, Z_2, \ldots, Z_N &\subseteq& \mathbb Z^{m} \; .
\end{array}
\end{eqnarray}
The coupling matrix $A$ and the vector $b$ are assumed to be given. In this context, the following blanket assumption is used.
\begin{assumption}
\label{ass::convexity}
The sets $X_1, X_2, \ldots, X_N \subseteq \mathbb R^{n}$ are non-empty convex polytopes, the sets $Z_1, Z_2, \ldots, Z_N \subseteq \mathbb Z^{m}$ are non-empty and compact, and the functions $f_i$ are convex on the convex hull of $X_i \times Z_i$.
\end{assumption}
The goal of this paper is to develop an efficient algorithm that finds $\varepsilon$-suboptimal points of~\eqref{eq::minlp}, which are defined as follows:
\begin{definition}
A feasible point $(x^\star,z^\star) \in X \times Z$ with $A x^\star = b$ is said to be an $\varepsilon$-suboptimal point of \eqref{eq::minlp}, with $\varepsilon > 0$, if
\begin{align*}
f(x^\star,z^\star) \leq f(x,z) + \varepsilon \,.
\end{align*}
for all $(x,z) \in X \times Z$ with $A x = b$.
\end{definition}

\begin{remark}
Instead of~\eqref{eq::minlp}, one could also consider more general optimization problems of the form
\begin{eqnarray}
\label{eq::minlp2}
\begin{array}{cl}
\underset{x \in X, z \in Z}{\min}& f(x,z) \\[0.3cm]
\text{s.t.}& A x + B z = b \; .
\end{array}
\end{eqnarray}
However, under mild regularity assumptions~\cite{Nocedal_2006}, this problem is equivalent to 
\begin{eqnarray}
\begin{array}{cl}
\underset{x \in X, y \in \cvx{Z}, z \in Z}{\min}& \sum_{i=1}^N \{ f_i(x_i,z_i) + \bar \lambda_i \| y_i - z_i  \|_1 \} \\[0.3cm]
\text{s.t.}& A x + B y = b \; .
\end{array}
\end{eqnarray}
with real-valued auxiliary variables $y$ and $L_1$-penalty parameters $\overline \lambda_i \gg 0$.
Here, $\cvx{Z}$ denotes the convex hull of $Z$.
Thus, for all theoretical purposes, it is sufficient to analyze problems of the form~\eqref{eq::minlp}, where only the real-valued variables are coupled.
\end{remark}

\begin{remark}
Modern MICP formulations, algorithms, and software can deal with rather general convex conic constraints~\cite{Lubin2017}. Such constraints are left out for simplicity of presentation. Nevertheless all results in this paper can be easily extended for general conic constraints, as long as they are separable. From a purely theoretical perspective, one might argue that this can always be achieved by adding suitable convex penalty functions to the functions $f_i$, because this paper makes no assumptions on the differentiability properties of $f$. However, more tailored, practical algorithms that could exploit the structures of particular conic constraints are beyond the scope of this paper. 
\end{remark}

\subsection{Notation}
We use the notation
\[
\partial_x g(x) = \left\{ a \in \mathbb R^n \mid \forall y \in \mathbb R^{n}, \; g(y) \geq g(x) + a^\tr ( y - x )  \right\}
\]
to denote the set of subgradients of a convex function $g: \mathbb R^{n} \to \mathbb R$ with respect to the variable $x$.

\section{Outer approximation}
\label{sec::OA}

\subsection{Polyhedral relaxations}
In order to construct polyhedral outer approximations of the epigraph of the objective function $f$ of~\eqref{eq::minlp}, we consider the auxiliary optimization problem
\begin{eqnarray}
\label{eq::AuxOpt}
f^\star(z) = \min_{x,y} \; f(x,z) \quad \mathrm{s.t.} \quad \left\{
\begin{array}{ll}
x = y & \mid \; \lambda \\[0.08cm]
A y = b \\[0.08cm]
y \in X \; .
\end{array} \right.
\end{eqnarray}
for a fixed integer parameter $z \in Z$. Here, $x$ and $y$ are real valued primal optimization variables and the notation ``$x = y \; \mid \lambda$'' is used to say that $\lambda$ denotes the dual variable that is associated with the constraint $x=y$.
\begin{proposition}
\label{prop::AuxOptDual}
If Assumption~\ref{ass::convexity} is satisfied, strong duality holds for~\eqref{eq::AuxOpt}, i.e., we have
\begin{eqnarray}
\label{eq::AuxOptDual}
f^\star(z) = \max_{\lambda} \; \min_{x,y} \; f(x,z) + \lambda^\tr (y-x) \quad \mathrm{s.t.} \quad \left\{
\begin{array}{l}
A y = b \\
\hspace{3mm} y  \in X
\end{array}
\right.
\end{eqnarray}
for all $z \in Z$.
\end{proposition}
\begin{proof}
See Appendix~\ref{app::AuxOptDual}.
\end{proof}
Let $x^\star(z),y^\star(z),\lambda^\star(z)$ denote any primal-dual solution of~\eqref{eq::AuxOpt} in dependence on $z$. By writing out the stationarity condition of~\eqref{eq::AuxOpt} with respect to $x$, we find that
\[
\lambda^\star(z) \in \partial_x f( x^\star(z), z ) \; ,
\]
i.e., $\lambda^\star(z)$ must be a subgradient of $f$ at the optimal solution of~\eqref{eq::AuxOpt}.
In order to understand the developments given below it is helpful to keep in mind that the reverse statement is not correct, i.e., a subgradient of $f$ at $(x^\star(z), z)$ is not necessarily a dual solution of~\eqref{eq::AuxOpt}. 


In contrast to the particular choice of the subgradient $\lambda^\star$ of $f$ with respect to $x$, the construction of a subgradient of $f$ with respect to $z$ is less critical for the construction of outer approximation methods. In the following, we assume that a function
\[
\mu^\star(z) \in \partial_z f( x^\star(z), z ) \; ,
\]
is given, which returns a subgradient of $f$ with respect to $z$ at the optimal solution of~\eqref{eq::AuxOpt}.
Because $f(x,z) = \sum_{i=1}^N f(x_i,z_i)$ is separable, the $i$-th block components, $\lambda_i^\star(z)$ and $\mu_i^\star(z)$, of the subgradients of $f$ are subgradients of $f_i$. Thus, the inequality
\[
f_i(x_i, z_i) \geq f_i^\star( \hat z ) + \left[ \lambda_i^\star(\hat z) \right]^\tr ( x_i  - x_i^\star(\hat z) ) + \left[ \mu_i^\star(\hat z) \right]^\tr ( z_i  - \hat z_i )
\]
holds for all $x_i \in X_i$, and $z_i \in Z_i$ and all $\hat z \in Z$. Here, the shorthand
$$f_i^\star( \hat z ) = f_i(x_i^\star( \hat z), \hat z_i)$$
is used. In this context, it is important to notice that the function $f_i^\star(\hat z)$ depends on the whole vector $\hat z$, not only on its $i$-th component, $\hat z_i$, because the equality constraints in~\eqref{eq::AuxOptDual} introduce a non-trivial coupling.
More generally, if $\Xi \subseteq Z_i$ denotes finite set of points in $Z$, we associate with $\Xi$ a set of hyperplane coefficients
\begin{eqnarray}
\mathcal H_i( \Xi ) = \left\{
(\alpha,\beta,\gamma) \left|
\begin{array}{l}
z \in \Xi \\[0.1cm]
\alpha = \lambda_i^\star(z) \\[0.1cm]
\beta = \mu_i^\star(z) \\[0.1cm]
\gamma = f_i^\star(z) - \alpha^\tr x_i^\star(z) - \beta^\tr z_i
\end{array}
\right.
\right\} \; .
\end{eqnarray}
Notice that this set of hyperplane coefficients defines a polyhedral outer approximation of the epigraph of $f_i$. Thus, these coefficients can be used to construct a piecewise affine lower bound on $f_i$, which is for all $(x_i,z_i) \in X_i \times Z_i$ given by
\begin{eqnarray}
\label{eq::separableLowerBounds}
 \Phi_{i}(x_i,z_i,\Xi) = \max_{(\alpha,\beta,\gamma) \in H_{i}(\Xi)} \left\{ \alpha^\tr x_i + \beta^\tr z_i + \gamma \right\} \; .
\end{eqnarray}
Finally, we can construct the function $\Phi(x,z,\Xi) = \sum_{i=1}^N \Phi_{i}(x_i,z_i,\Xi)$. This function is---by construction---a piecewise affine lower bound on $f$,
\begin{eqnarray}
\label{eq::LowerBoundF}
\forall (x,z) \in X \times Z, \qquad \Phi(x,z,\Xi) \leq f(x,z) \; .
\end{eqnarray}
Next, our particular choice of the subgradient $\lambda^\star(z)$ of $f$ as the dual solution of~\eqref{eq::AuxOpt} enables us to establish the following tightness property of the affine lower bound $\Phi$.
\begin{lemma}
\label{lemma::lowerBound}
Let $\Xi \subseteq Z$ be any finite set of points. If Assumption~\ref{ass::convexity} holds, then the equation
\begin{eqnarray}
\label{eq::TightLowerBound}
f^\star(z) = \min_{x \in X} \; \Phi(x,z,\Xi) \quad \mathrm{s.t.} \quad A x = b \; .
\end{eqnarray}
holds for all $z \in \Xi$.
\end{lemma}
\begin{proof}
See Appendix~\ref{app::lowerBound}.
\end{proof}

\begin{remark}
\label{rem::Relaxation}
If the set $\Xi$ consists of $m$ points, the computational cost for constructing the lower bound~\eqref{eq::LowerBoundF} of $f$ has order $\order(m N)$. Notice that if we would have ignored the separable structure of $f$, the computational cost of computing the same lower bounding function would have been of order $\order(m^N)$. Thus, the construction of~\eqref{eq::LowerBoundF} as a sum of the lower bounds of the separable objective function is much cheaper than a direct construction of lower bounds of $f$. This reduction in complexity has for the first time been observed and exploited by Hijazi and co-workers~\cite{Hijazi2014}. By now, the expoitation of separability via extended formulations can be considered as a standard that has been adopted in many modern MICP algorithms and software tools~\cite{Kilinc2017,Lubin2017}.
\end{remark}

\subsection{Outer approximation algorithm}
\label{sec:OA_alg}
Algorithm~1 outlines the main steps of the outer approximation algorithm. Notice that this algorithm basically coincides with the original outer approximation algorithm that has been proposed in~\cite{Duran_1986}. The only notable differences of Algorithm~1 compared to traditional OA are that the MILP in Step~3 uses the extended formulation based outer approximation variant from~\cite{Hijazi2014}. Moreover, because we do not assume that $f$ is differentiable, we have to use the particular choice, $\lambda^\star(z)$, of the subgradient, which is found as the dual solution\footnote{The idea to use dual solutions as subgradients for the construction of polyhedral outer approximation is not new and can---in a very similar setting---be found in~\cite{Lubin2017}.} of~\eqref{eq::AuxOpt}.
\begin{figure*}[h] \label{alg:basic_OA}
\footnotesize
\begin{center}
 \begin{minipage}{0.99\textwidth}
\rule{1\textwidth}{0.3mm}\\
\textbf{Algorithm~1: Outer Approximation for MICP}\\
\rule{1\textwidth}{0.3mm}\\
\textbf{Input:} Initial guess $z \in Z$ and a numerical tolerance $\varepsilon > 0$.\\

\textbf{Initialization:} Set $\Pi = \varnothing$ and $U = \infty$.\\

\textbf{Repeat:}\\[-0.3cm]
\begin{enumerate}
\setlength{\itemsep}{4pt}

\item 
Solve the convex optimization problem
\begin{eqnarray}
\label{eq::NLP}
f^\star(z) = \min_{x,y} \; f(x,z) \quad \mathrm{s.t.} \quad \left\{
\begin{array}{ll}
x = y & \mid \; \lambda \\[0.08cm]
A y = b \\[0.08cm]
y \in X \; .
\end{array} \right.
\end{eqnarray}

\item If~\eqref{eq::NLP} has no feasible solution, return a certificate of infeasibility. Otherwise, update
$$U \leftarrow \min \left\{ U, \, f^\star(z) \right\} \quad \text{and} \quad
\Pi \leftarrow \Pi \cup \{ z \} \; .$$

\item 
Solve the (extended) MILP
\begin{eqnarray}
\label{eq::MILP1}
\begin{array}{rcl}
(x^{+},y^+,z^+) \in & \underset{x \in X, y, z \in Z}{\mathrm{argmin}} & \sum_{i=1}^N y_i \\[0.2cm]
&\text{s.t.}&
\left\{
\begin{array}{l}
\forall i \in \{ 1, \ldots, N \}, \\[0.1cm]
\forall (\alpha_i, \beta_i, \gamma_i) \in \mathcal H_i \left( \Pi \right) \\[0.1cm]
\alpha_i^\tr x_i + \beta_i^\tr z_i + \gamma_i  \leq y_i \\[0.1cm]
A x = b
\end{array}
\right.
\end{array}
\end{eqnarray}

\item If $U - \sum_{i=1}^N y_i^+ \leq \varepsilon$, terminate.

\item 
Update $z \leftarrow z^+$ and go to Step~1.

\end{enumerate}

\removelastskip\rule{1\textwidth}{0.3mm}
\end{minipage}
\end{center}
\end{figure*}
Notice that Step~1 of Algorithm~1 solves~\eqref{eq::minlp} under the additional constraint that the integer $z$ is fixed. This implies that
$$f(x^\star,z) \geq V^\star$$
is an upper bound on the optimal objective value $V^\star$ of~\eqref{eq::minlp}. Thus, the current upper bound $U$ can be updated in Step~2. Moreover, the MILP~\eqref{eq::MILP} is (by construction) equivalent to solving the relaxed optimization problem\footnote{Since the inception of the idea of Gomory cuts in the 1960s~\cite{Gomory1960}, cutting plane methods for MILP have evolved significantly. Nowadays, there exist efficient algorithm and solvers for MILP and we refer to~\cite{Conforti2009} for an overview.}
\[
\min_{x \in X,z \in Z} \; \Phi(x,z,\Pi) \quad \mathrm{s.t.} \quad Ax = b \; ,
\]
which implies that
$$\sum_{i=1}^N y_i^+ \leq V^\star$$
is a lower bound on $V^\star$. Thus, the difference, $U - \sum_{i=1}^N y_i^+$, between the current upper and lower bounds can be used as a termination criterion, which is implemented in Step~3 of Algorithm~1. The following finite termination result for outer approximation is (at least in very similar versions) well-known in the literature~\cite{Duran_1986,Lubin2017}.

\begin{theorem}
\label{thm::termination1}
If Assumption~\ref{ass::convexity} is satisfied, then Algorithm~1 terminates after a finite number of iterations.
\end{theorem}
\begin{proof}
See Appendix~\ref{app::termination1}.
\end{proof}

\begin{remark}
Algorithm~1 uses Hijazi's extended formulation~\cite{Hijazi2014} for constructing the MILPs~\eqref{eq::MILP}, which arguably exploit separability of the objective function to some extent. However, Algorithm~1 is not a fully distributed algorithm. In fact, a major disadvantage of Algorithm~1 becomes apparent, if one considers the special case that the constraint $A x = b$ happens to be redundant. In this case, the optimal solution of~\eqref{eq::minlp} could have been found with much less effort by solving the separable MICPs,
\[
\min_{x_i \in X_i,z_i \in Z_i} f_i(x_i,z_i)
\]
which have much fewer integer variables. However, if Algorithm~1 is applied to such a problem with redundant equality constraint, this property is not detected and a large number of large-scale NLPs and large scale MILPs might have to be solved instead, until convergence is achieved. The goal of this paper is to mitigate this deficiency of Algorithm~1 by proposing a partially distributed outer approximation algorithm that exploits the structure of the separable objective in a better way.
\end{remark}

\section{Partially distributed outer approximation algorithm}
\label{sec::PADOA}

This section introduces a partially distributed outer approximation optimization algorithm for finding $\varepsilon$-suboptimal solutions of~\eqref{eq::minlp}.

\subsection{Partially decoupled upper bounds}
The main idea of many distributed convex and local optimization methods is to solve a set of smaller-scale decoupled optimization problems in place of a single large one~\cite{Benders_1962,Boyd_2011,Dantzig_1960}. Similarly, consider partially decoupled optimization problems of the form
\begin{eqnarray}
\label{eq::pdIntro}
V_k(z) = &\underset{x,y,\zeta_k}{\min}& f_k(x_k,\zeta_k) + \Psi_k(x,z) \\[0.2cm]
&\text{s.t.}& \left\{ 
\begin{array}{ll}
x = y & \mid \; \lambda \\
A y = b \\
y \in X \\
\zeta_k \in Z_k
\end{array}
\right.
\end{eqnarray}
for $k \in \{ 1, \ldots, N \}$. Here, the integer variable $z \in Z$ is regarded as a fixed parameter and only the much smaller dimensional integer vector $\zeta_k \in Z_k$ is optimized. However, concerning the real-valued variables, the whole vector $x \in X$ is kept as an optimization variable. In this context, the shorthands
\[
\forall (x,z) \in X \times Z, \qquad \Psi_k(x,z) = \underset{j \neq k}{\sum} f_j(x_j,z_j)
\]
are introduced in order to keep the $x$-dependence of the remaining summands, i.e., all objective terms whose index is not equal to $k$. As in the previous section, $\lambda$ denotes the dual solution that is associated with the consensus constraint ``$x = y$''. Because the constraint $\zeta_k \in Z_k$ enforces integrality, strong duality of~\eqref{eq::pdIntro} does not hold in general. However, if $\zeta_k^\star(z)$ denotes an optimal solution of~\eqref{eq::pdIntro} for the integer variable and if Assumption~\ref{ass::convexity} holds, we still have
\begin{eqnarray}
V_k(z) = \max_{\lambda} &\underset{x,y}{\min}& \; \; f_k(x_k,\zeta_k^\star(z)) + \Psi_k(x,z) + \lambda^\tr(y-x) \notag \\
\label{eq::maxDualDef}
&\text{s.t.}& \; \; \left\{ 
\begin{array}{l}
A y = b \\
y \in X \; .
\end{array}
\right.
\end{eqnarray}
The proof of this statement is completely analogous to Proposition~\ref{prop::AuxOptDual}, i.e., if the linear coupling constraint $Ax =b$ has a solution in $X$, a maximizer of~\eqref{eq::maxDualDef} exists and can be used to define a suitable subgradient. Also note that the functions $V_k$ yield upper bounds on the objective value of~\eqref{eq::minlp},
\begin{align}
\label{eq::UpperBound1}
\forall z \in Z, \qquad \min_{k} \, V_k(z) \geq V^\star
\end{align}
At this point, it should be mentioned that one basic assumption of the algorithmic developments in this paper is that the complexity of the mixed integer optimization problems of interest depends mostly on the number of integer variables. This is in contrast to the number of real-valued variables, which may be assumed to have a negligible influence on the overall complexity of the mixed-integer optimization problem. In other words, we assume that~\eqref{eq::pdIntro} is much easier to solve than~\eqref{eq::minlp} in the sense that it contains much fewer integer variables, although both problems have the same number of real-valued variables. Here, it is important to keep in mind that, although the algorithmic developments in this paper are inspired by the field of distributed optimization, the algorithm in this paper is (at least in the form in which we present and analyze it) not fully distributed. This is because solving~\eqref{eq::pdIntro} requires the evaluation of the function $\Psi_k$, which, in turn, requires the evaluation of all functions $f_j$ with $j \neq k$.

\subsection{Partially decoupled lower bounds}
In this paper, we suggest to solve the decoupled MICPs~\eqref{eq::pdIntro} by lower level solvers that implement the traditional outer approximation algorithm that has been reviewed in Section~\ref{sec:OA_alg}. Notice that if Assumption~\ref{ass::convexity} is satisfied, strong duality holds, i.e., these lower level solvers will return piecewise affine models
\[
\Theta_k^\star: X \times Z_k \to \mathbb R \; ,
\]
which must satisfy the condition
\begin{eqnarray}
\label{eq::LowerLevelTermination}
V_k(z) - \epsL \; \leq \min_{x \in X,\zeta \in Z_k} \Theta_k^\star(x,\zeta) \quad \mathrm{s.t.} \quad A x = b 
\end{eqnarray}
upon termination. Here, $\epsL \geq 0$ denotes the numerical tolerance of the lower level OA solvers. Notice that the optimization problem on the right hand of~\eqref{eq::LowerLevelTermination} corresponds to the last MILP relaxation that is solved by the lower level OA solver. In practice the function $\Theta_k^\star$ can be stored by maintaining a set of hyperplane coefficients as explained in detail in the previous section.
Moreover, in order to avoid the accumulation of too many hyperplanes, one can discard all hyperplanes that are inactive at the optimal solution of the last MILP relaxation, because this operation does not affect the right hand expression of~\eqref{eq::LowerLevelTermination}.

The main idea of partially distributed outer approximation is to communicate the piecewise lower bounding functions $\Theta_k^\star$ to a central coordinator, who constructs a piecewise affine lower bound on the function $f$, solves a master MILP problem, and updates $z$. Here, one option is use the maximum over the function $\Theta_k^\star$ in order to obtain the lower bound
\begin{align}
\forall x \in X, \; \forall z \in Z, \qquad   \max_{k} \; \Theta_k^\star(x,z_k) \; \leq \; f(x,z) \; .
\end{align}
However, in order to arrive at a practical implementation, it is recommendable to further refine this bound. This can be done by maintaining a collection of integers, $\Pi \subseteq Z$, such that the function
\begin{align}
\Theta(x,z) = \max \left\{ \; \Phi(x,z,\Pi) \; , \; \max_{k} \; \Theta_k^\star(x,z_k) \; \right\} \; ,
\end{align}
can be used as a piecewise affine lower bound on $f$. Recall that the function $\Phi$, which has been introduced in the previous section, exploits the separability properties of $f$. The integer collection $\Pi$ is then maintained by updating
\[
\Pi \leftarrow \Pi \cup \{ \zeta^\star \} \; ,
\]
where $\zeta^\star = [ \zeta_1^\star, \zeta_2^\star, \ldots, \zeta_N^\star ]$ is an integer vector, whose components are optimal solutions for the integer variables of the partially decoupled problems~\eqref{eq::pdIntro}.

\subsection{Partially distributed outer approximation (PaDOA)}

\begin{figure*}[t] 
\footnotesize
\begin{center}
 \begin{minipage}{0.99\textwidth}
\rule{1\textwidth}{0.3mm}\\
\textbf{Algorithm~2: Partially Distributed Outer Approximation (PaDOA)}\\
\rule{1\textwidth}{0.3mm}\\
\textbf{Input:} Initial guess $z \in Z$ and a numerical tolerance $\varepsilon > 0$.\\

\textbf{Initialization:} Set $\Pi = \varnothing$, $\Theta(\cdot,\cdot) = {-\infty}$, and $U = \infty$.\\

\textbf{Repeat:}\\[-0.3cm]
\begin{enumerate}
\setlength{\itemsep}{4pt}

\item \label{step::pdminlp} Solve for all $k \in \{ 1, \ldots, N\}$ the partially decoupled MICPs
\begin{eqnarray}
\label{eq::pdminlp}
\begin{array}{rcl}
V_k(z) =& \underset{x,y,\zeta_k}{\min} & f_k(x_k,\zeta_k) + \Psi_k(x,z) \qquad \text{with} \qquad \Psi_k(x,z) = \underset{j \neq k}{\sum} f_j(x_j,z_j) \; . \\[0.2cm]
&\text{s.t.}& \left\{
\begin{array}{ll}
x = y & \mid \; \lambda \\
A y = b \\
y \in X \\
\zeta_k \in Z_k
\end{array} \right.
\end{array}
\end{eqnarray}
If \eqref{eq::pdminlp} is infeasible, terminate and return a certificate of infeasibility. Otherwise, update the set $\Pi \leftarrow \Pi \cup \{ \zeta_k^\star \}$ and construct a piecewise affine model $\Theta_k^*$ such that condition \eqref{eq::LowerLevelTermination} is satisfied.

\item \label{step::Terminate} Update the upper bound $U \leftarrow \min \left\{ U, \, V_1(z), \dots , V_N(z) \right\}$ and construct the piecewise lower bounding function $\Phi(x,z,\Pi)$ as in \eqref{eq::separableLowerBounds}.

\item Update the lower bound
\[
\forall x \in X, \, \forall z \in Z, \qquad \Theta(x,z) \leftarrow \max \left\{ \; \Theta(x,z) \; , \; \Phi(x,z,\Pi) \; , \; \max_{k} \; \Theta_k^\star(x,z_k) \; \right\} \; 
\]

\item \label{step::MILP} Solve the MILP problem
\begin{eqnarray}
\label{eq::MILP}
(x^+,z^+) \in \underset{x \in X, z \in Z}{\text{argmin}} & \Theta(x,z) \quad \mathrm{s.t.} \quad 
A x = b
\end{eqnarray}

\item \label{step::Return} If $U - \Theta(x^+,z^+) \leq \varepsilon$, terminate. Otherwise, update $z \leftarrow z^+$ and go to Step~1.

\end{enumerate}

\removelastskip\rule{1\textwidth}{0.3mm}
\end{minipage}
\end{center}
\end{figure*}

Algorithm~2 outlines a partially distributed algorithms for solving~\eqref{eq::minlp}. There are four main steps. In the first step, the partially decoupled MICPs of the form~\eqref{eq::pdIntro} are solved by using a traditional outer approximation method. Under the assumption that the original MICP~\eqref{eq::minlp} is feasible, the partially decoupled MICPs are feasible, too. Thus, the outer approximation solvers will return optimal integer solutions $\zeta_k^\star$ and associated piecewise affine lower bounds $\Theta_k^\star$ such that~\eqref{eq::LowerLevelTermination} is satisfied. The second step of Algorithm~2 updates the associated upper bound $U$ based on the inequality~\eqref{eq::UpperBound1} as well as the piecewise affine lower bound. In practice, this step is implemented by storing the union of all supporting hyperplane coefficients that are needed to represent $\Theta$.
Finally, the third step of Algorithm~2 solves a large scale MILP problem. This MILP is constructed in analogy to the corresponding step in the traditional outer approximation algorithm. It yields a lower bound,
\[
\Theta(x^+,z^+) \leq V^\star \; ,
\]
on the objective value $V^\star$ of~\eqref{eq::minlp}. Thus, the difference between the current upper and lower bounds,
\[
U - \Theta(x^+,z^+) \; ,
\]
can be used as a termination criterion, which is implemented in the fourth step of Algorithm~2. If the termination is not successful, the integer variables $z$ are updated, and the algorithm subsequently proceeds to the next iteration.

Notice that the main difference between Algorithms~1 and~2 is the introduction of partially decoupled MICP problems that can be solved separately and which contain much fewer integer variables than the orginal MICP~\eqref{eq::minlp}. The theoretical results in Section~\ref{sec::convergenceAnalysis} will elaborate further on the benefits of this alternation strategy. Moreover, in Section \ref{sec:sim} a numerical case study is examined, which illustrates the practical advantages of Algorithm~2.

\subsection{Relation to distributed local optimization methods}
The idea to ``augment'' the local objective functions $f_i$ with a suitable function $\Psi_i$ is frequently used in the context of distributed local optimization algorithms. For example, in the context of dual decomposition~\cite{Everett1963,Necoara2008}, one augments the separable functions $f_i$ with linear functions of the form\footnote{In the context of convex optimization Problem~\eqref{eq::minlp} is considered without integer variables $z$---this is why the functions $\Psi_i$ depend in the convex case on $x$ only.}
$$\Psi_i(x) = \sigma^\tr A x \; ,$$
where $\sigma$ is the current dual iterate. Similarly, in the context of ADMM or ALADIN, one uses augmented Lagrangians~\cite{Andreani2007,Powell1969}, as in
\begin{align}
\label{eq::augmentedLagrangian}
\Psi_i(x,y) = \sigma^\tr A x + \frac{\rho}{2} \Vert x - y \Vert^2 \; ,
\end{align}
where $z$ and $\sigma$ are the current primal and dual iterates; see~\cite{Boyd_2011,Eckstein1992,Houska_2016}. In fact, the construction of Algorithm~2 is inspired by the distributed local nonlinear programming method ALADIN. Here, we recall that ALADIN alternates between solving small-scale decoupled NLPs that are augmented by terms of the form~\eqref{eq::augmentedLagrangian} and large scale equality constrained quadratic programming problems that update $\sigma$ and $y$~\cite{Houska_2016}. This is in analogy to Algorithm~2, which alternates between solving decoupled MICPs (Step 1) and large-scale coupled MILPs (Step 3). However, unlike ALADIN, augmented Lagrangians are not used in Algorithm~2 as Lagrange multipliers in integer programming are not related to sensitivity and generally not applicable.

Also note that the construction of the functions $\Psi_i$ in Algorithm~2 also has similarities with Gauss-Seidel or more general block-coordinate descent methods~\cite{Tseng2001,Wright2015} in the sense that a partial decoupling is obtained by fixing some of the integer variables while others are optimized. However, despite all these analogies and similarities of Algorithm~2 with methods from the field of local and convex optimization, we would like to highlight that all these existing distributed optimization methods are not reliably applicable \cite{Takapoui_2016}.

\subsection{Convergence Analysis}
\label{sec::convergenceAnalysis}


In this section we provide a concise overview of the convergence 
properties of Algorithm~2.
The following theorem establishes one of the main results of this paper, namely, that Algorithm~2 converges after one iteration if the integer iterate, $z$, is initialized with an optimal solution of~\eqref{eq::minlp}. This is contrast to Algorithm~1, which does not necessarily terminate after a small number of steps---not even if it is initialized at an optimal solution.

\begin{theorem}
\label{thm::termination}
Let Assumption~\ref{ass::convexity} be satisfied and let $(x^\star,z^\star)$ be a minimizer of~\eqref{eq::minlp}. If Algorithm~2 is initialized with $z = z^\star$ and if the termination tolerances of the lower level solvers satisfy $\epsL \leq \epsilon$, then the termination criterion in Step~4 is satisfied. In other words, the algorithm terminates after one step.
\end{theorem}
\begin{proof}
See Appendix~\ref{app::termination}.
\end{proof}


Notice that the statement of the above theorem is of fundamental relevance and a very favorable property of PaDOA. If we work with other global optimization methods, say branch-and-bound, an empirical observation is that such existing global optimization algorithm often find a global solution early on but then keep on iterating until the lower bound is accurate enough to prove global optimality. In contrast to this, PaDOA terminates as soon as a global minimizer is added to the collection $\Pi$. In fact, Theorem~\ref{thm::termination} implies that global optimality of a point $z^\star \in Z$ can be verified by solving the $N$ instances of the partially decoupled MICPs and the master MILP~\eqref{eq::MILP}. Notice that this result is not in conflict with existing results from the field of complexity theory, because the master MILP~\eqref{eq::MILP} remains NP-hard~\cite{Garey1979,Murty1987}.

\begin{remark}
The result of Theorem~\ref{thm::termination} relies heavily on the convexity of the functions $f_i$ on the convex hull of $X_i \times Z_i$, although this fact is not highlighted explicitly in the proof. This convexity assumption is first of all required implicitly by our assumption that the lower level solvers return piecewise level models $\Theta_k$, which satisfy the termination condition~\eqref{eq::LowerLevelTermination} (this assumption is only reasonable if strong duality holds) and which need to be global lower bounds on $f$. These properties are in general all not satisfied if one considers more general non-convex MINLPs. 
\end{remark}

The following theorem establishes the fact that Algorithm~2 converges after a finite number of iterations under exactly the same conditions under which convergence of Algorithm~1 can be established.

\begin{theorem}\label{thm::PaDOA}
Let Assumption~\ref{ass::convexity} be satisfied.
If the termination tolerances of the lower level solvers satisfy $\epsL \leq \epsilon$, then
Algorithm~2 terminates after a finite number of steps (independently of the initialization).
\end{theorem}
\begin{proof}
See Appendix~\ref{app::PaDOA}.
\end{proof}

\section{Implementation and Case Study} \label{sec:sim}

The Partially Distributed Outer Approximation method is implemented in MATLAB R2017b.
The optimization subproblems are solved using Gurobi~\cite{GUROBI2009}
implemented via CasADi~v1.9.0~~\cite{Andersson_2018}. All numerical experiments
were run on a 2.9GHz Intel Core i5-4460S CPU with 8GB of RAM.


\subsection{Problem Description}
An important problem in the planning and operation of a heating and/or cooling system is the scheduling of so-called Thermostatically Controlled Loads (TCLs) \cite{Kohlhepp_2017}. These are devices that are used to regulate the temperature of a room/building within a certain user-defined interval known as a ``deadband''. The optimal operation strategy is especially difficult to determine when a non-constant cost function is introduced for a population of heterogeneous TCLs \cite{Zhang2012}. The cost function may represent the cost of electricity or user-discomfort from noise generation. Regardless, such devices typically only have an ``on'' and an ``off'' setting and thus the resulting scheduling problem can be formulated as a binary MIP as seen in \eqref{eq:tcl} for $R$ regions with a finite time horizon $H$. The equations in \eqref{eq:tcl} are based on the formulation given in \cite{Koch_2011}, but with dynamics modeling the interaction between each region and a linear cost function instead of a quadratic.

\begin{subequations} \label{eq:tcl}
	\begin{align} 
	\label{eq:SepProbObj}
	&\min\limits_{T(\cdot),u(\cdot)} \sum\limits_{t=0}^{H-1} c(t)u(t) + \gamma (T_i(t) - T_{ref}(t))^2, \\
	\label{eq:SepProbEqcnstr2}
	&\text{subject to} \hspace{2mm} \forall i\in \{1,\dots,R\}, \nonumber\\
	& \underline{T}_i \leq T_i(t) \leq \overline{T}_i,  \quad \forall t\in \{0,\dots,H\}  \\
	\label{eq:SepProbIneqcnstr2} 
	&u_i(t)\in \{0,1\}, \hspace*{.75cm} \forall t\in \{0,\dots,H-1\} \\
	&\forall t\in \{0,\dots,H\},\nonumber\\
	\label{eq:consConstr}
	&T_i(t+1)=T_i(t)+b_i u_i(t)+a_i\left(\frac{T_i(t)+T_{amb}(t)+\sum_{j \in N(i)}T_j(t)}{|N(i)|+2} - T_i(t)\right),  
	\end{align}
\end{subequations}

\noindent where $c(t)$ is the vector of device costs at time $t$, $\gamma$ is a comfort parameter, $\underline{T}_i$ and $\overline{T}_i$ are the deadband temperature limits of device $i$, $a_i$ and $b_i$ are heat transfer parameters, $T_{amb}(t)$ is the ambient temperature at time $t$ and $N(i)$ are the number of regions neighbouring $i$. Equation \eqref{eq:consConstr} models the thermodynamics of each room in a simplified manner, i.e., it takes an average of the current and surrounding temperatures to update the temperature of the next time step. 
This formulation results in $H+1$ real-valued and $H$ binary variables per region. Figure \ref{fig:rooms} shows two possible initial configurations of \eqref{eq:tcl}. The ambient temperature is taken from \cite{WetterDienst} for two days in June 2017 in the Karlsruhe (Germany) area. High prices of \$25.67/kW are set from 2pm to 8pm (time steps 6 to 12 and 29 to 35) with low and medium prices of \$2.46/kW and \$4.62/kW in all other time steps. Each region is initialized at 20 degrees with $a_i=0.2$ and $b_i=-2$.

\begin{figure}[ht]\label{fig:rooms}
	\begin{center}
		\begin{tabular}{ccc}
			\includegraphics[width = 0.5\textwidth]{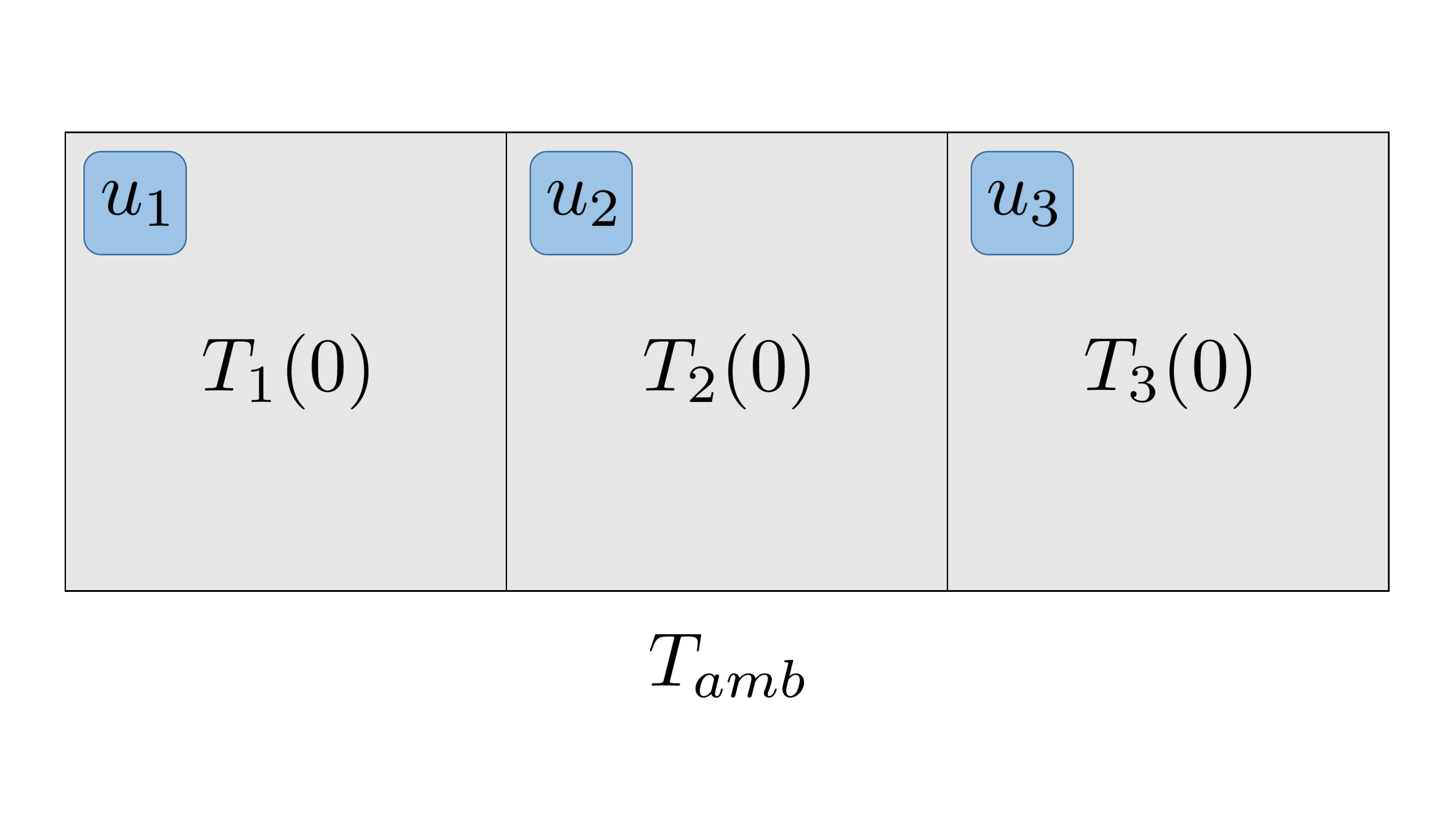}
			&
			\includegraphics[width = 0.5\textwidth]{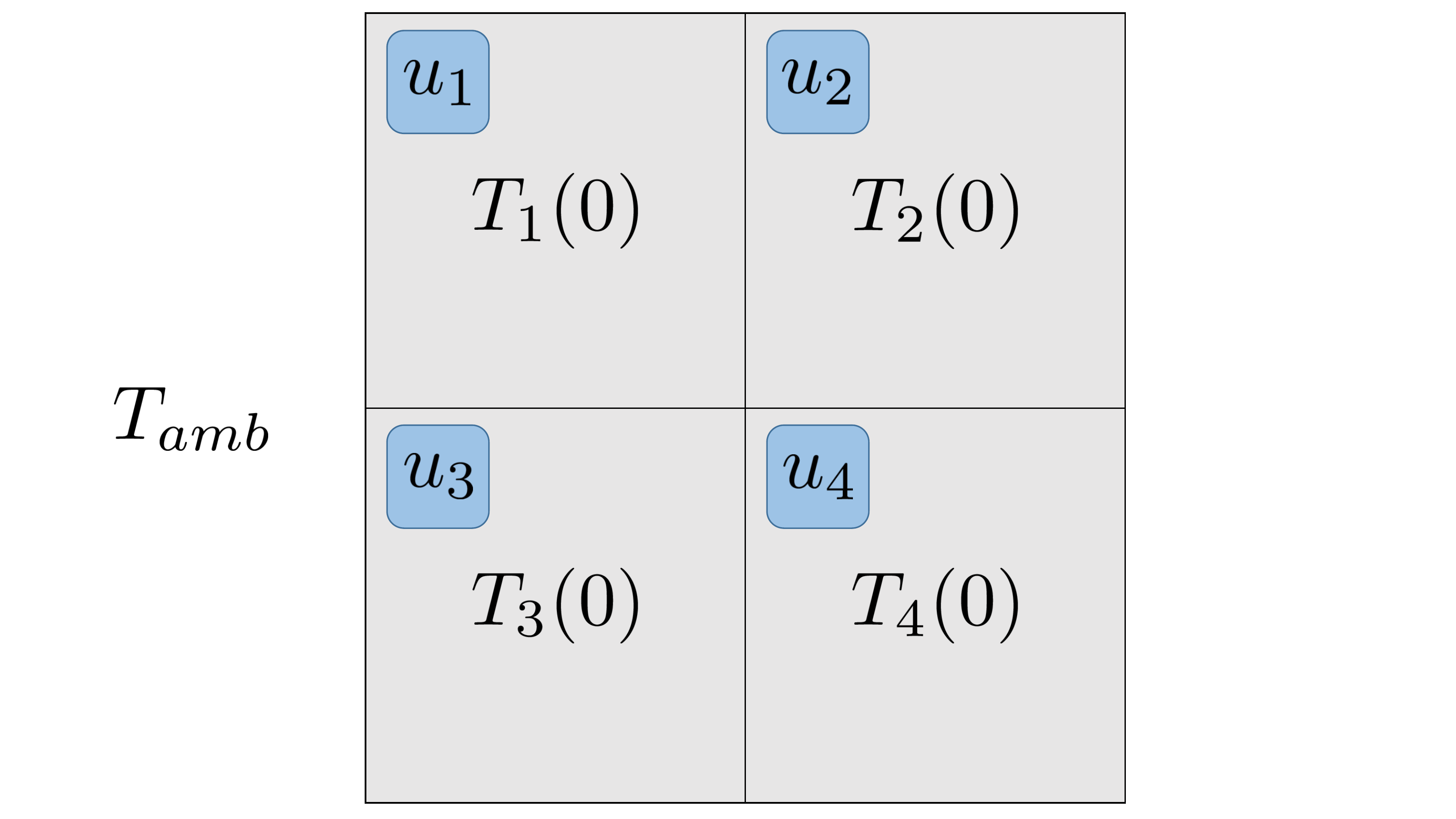}	
		\end{tabular}
	\end{center}
	\caption{Two room configurations with controlled cooling elements $u_i$, ambient temprature $T_{amb}$ and initial temperatures $T_i(0)$.}
\end{figure}

\subsection{Results for MILP}\label{sec:milp_sim}

If the comfort parameter $\gamma$ is taken to be zero then Problem \eqref{eq:tcl} is linear and separable but coupled in both its discrete and real-valued variables. Shown in Tables \ref{t:3room} and \ref{t:4room} are the simulation results for each configuration, respectively. The results of Algorithm~2 are compared with results obtained from a Branch and Bound approach as implemented in Bonmin with default settings \cite{Bonami_2008} as well as the commercial MIQP solvers Gurobi and CPLEX~\cite{CPLEX2009}. An example solution for the 3 room case is depicted in Figure \ref{fig:params}.

\begin{figure}[th!]\label{fig:params}
	\begin{center}
		\begin{tabular}{ccc}
			\includegraphics[width = 0.9\textwidth]{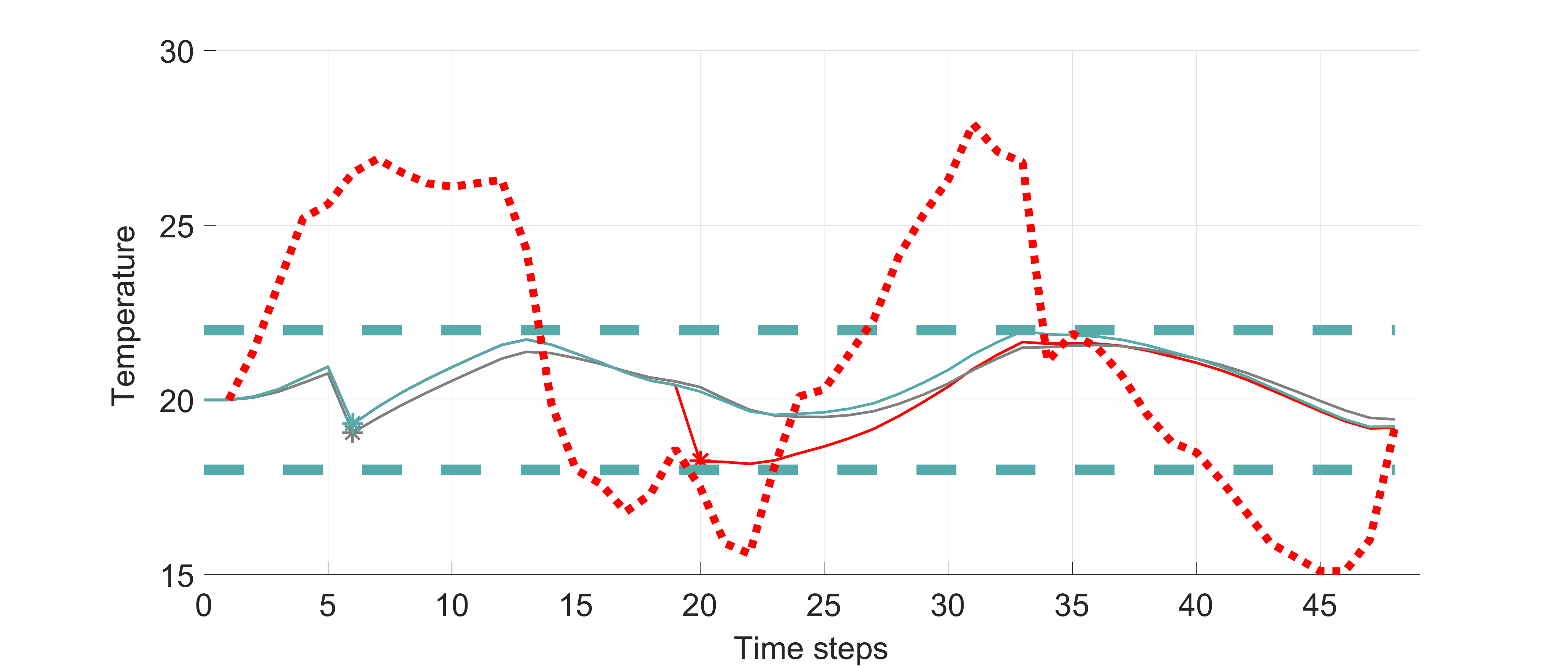}	
		\end{tabular}
	\end{center}
	\caption{The red dotted line is the ambient temperature, the blue dotted lines are the limits of the temperature deadzone and the solid lines are the temperature trajectories of each region in the three room scenario. Highlighted on the trajectories are points where the coolers are activated.}
\end{figure}


\begin{table}[ht]	
	\setlength{\tabcolsep}{4pt}
	\renewcommand{\arraystretch}{1.3}
	\begin{center}
		\caption{Results obtained for the TCL problem with a 3-room configuration.}
		\label{t:3room}
		\begin{tabular}{|c|c|c|c|c|c|}
			\hline
				& Time steps:& 8 & 24  & 48  & 62 \\
			\hline
					& obj. 		&0	 &13.86	&16.32	&21.23	\\
			Alg.~2	& time(s) 	&0.15&0.20	&0.89	&5.23	\\
					& iter. 	&2	&1	&2	&2	\\
			\hline
			B-OA	&obj. 		&0	 &13.86	&16.32	&21.23	\\
					&time(s) 	&0.15&31.06	&1,345	&52,395	\\
			\hline
			B-BB	&obj. 		&0	 &13.86	&16.32	&21.23	\\
					&time(s) 	&0.15	&29.19	&480.27	&828.08	\\
			\hline
			Gurobi	&obj. 		&0	 &13.86	&16.32	&21.23	\\
					&time(s) 	&0.22	&0.56	&0.85	&4.39	\\
			\hline
			CPLEX   &obj. 		&0	 &13.86	&16.32	&21.23	\\
					&time(s) 	&0.07	&0.16	&0.54	&2.81	\\
			\hline
		\end{tabular}
	\end{center}
\end{table}


\begin{table}[ht]	
	\setlength{\tabcolsep}{4pt}
	\renewcommand{\arraystretch}{1.3}
	\begin{center}
		\caption{Results obtained for the TCL problem with a 4-room configuration.}
		\label{t:4room}
		\begin{tabular}{|c|c|c|c|c|c|}
			\hline
			& Time steps:& 8 & 24  & 48  & 62 \\
			\hline
			& obj. 		&0	 &9.24	&9.24	&11.69	\\
			Alg.~2	& time(s) 	&0.17&0.25	&0.39	&1.41	\\
			& iter. 	&2	&2	&2	&2	\\
			\hline
			B-OA	&obj. &0	 &9.24	&9.24	&11.69	\\
			&time(s) 	&0.14&28.34	&48.42	&1,610	\\
			\hline
			B-BB	&obj. 	&0	 &9.24	&9.24	&11.69	\\
			&time(s) 	&0.18	&33.11	&56.83	&822.48	\\
			\hline
			Gurobi	&obj. 	&0	 &9.24	&9.24	&11.69	\\
			&time(s) 	&0.22	&0.37	&0.51	&1.08	\\
			\hline
			CPLEX   &obj. 	&0	 &9.24	&9.24	&11.69	\\
			&time(s) 	&0.07	&0.09	&0.11	&0.48	\\
			\hline
		\end{tabular}
	\end{center}
\end{table}

\begin{table}[ht]	
	\setlength{\tabcolsep}{4pt}
	\renewcommand{\arraystretch}{1.3}
	\begin{center}
		\caption{Results obtained for the TCL problem with a linear room configuration.}
		\label{t:nroom}
		\begin{tabular}{|c|c|c|c|c|c|c|c|c|}
			\hline
			& Rooms:	 & 7 & 7  & 7  & 10  & 12 & 18 & 20 \\
			\hline
			& Time steps:& 8 & 24 & 48 & 48 & 36 & 24 & 24\\
			\hline
			& obj. 		&0	 &23.1	&28.02	 & 37.26 & 41.88 & 50.82 & 55.44	\\
			Alg.~2	& time(s) 	&0.15&1.03		&4,533.4&9,127.5&8,294&853.7&16,658	\\
			& iter. 	&2	&2	&2	&2&2&2&2\\
			\hline
			B-OA		&obj. 		&0	  &23.1  & N/A& N/A& N/A& N/A& N/A	\\
			&time(s) 	&0.16 & 5134 & N/A& N/A& N/A& N/A& N/A	\\
			\hline
			B-BB	&obj. 	&0	 &23.1	&28.02	  & 37.26 & 41.88 & 50.82 & 55.44	\\
			&time(s) 	&0.18	&690.23 & 1,682.5  &3,076.0&3,261.5&2,999.1&4,658.4	\\
			\hline
			Gurobi		&obj. 		&0	 	&23.1	&28.02  &N/A & 41.88 & 50.82 & N/A	\\
			&time(s) 	&0.25   &0.93	&56.09  &N/A & 56,842 & 31,926 & N/A	\\
			\hline
			CPLEX   &obj. 		&0	 &23.1	&28.02	 & 37.26 & 41.88 & 50.82 & N/A	\\
			&time(s) 	&0.08 & 0.65 & 150.31& 13,416 & 12,885 & 57,108 & $>245,000$	\\
			\hline
		\end{tabular}
	\end{center}
\end{table}

At first glance, the results from Tables \ref{t:3room} and \ref{t:4room} may seem surprising since the 4 room case has more space to keep cool but nonetheless is able to do so at a lower cost than the 3 room case. This is due to an insulation effect that the 4 room configuration enjoys. With the activation of two coolers in the first six time steps, the room temperatures can stay within their deadbands for the entire 48 hour period. In contrast, the 3 room configuration is more susceptible to the ambient temperature and requires more use of the coolers. This also seems to have increased the computational complexity of the problem and requires more time for the 3 room case to be solved than the 4 room case. It should be noted that several initializations were tested and the solution times were not significantly affected, implying that this was not the cause of the runtime differences in the two cases.


One of the advantages of using a distributed method is the ability to solve problems that would be otherwsie intractable for a centralized solver. Tables \ref{t:3room} and \ref{t:4room} show results for cases containing up to 496 variables, but even larger problems may be considered. Table \ref{t:nroom} shows results for a variety of time horizons and rooms. Here, the room configuration is instead arranged such that the rooms are in a line. While unrealistic for most buildings, this setup is realistic for the temperature control of a train or rooms next to a corridor. Mathematically, this example differs somewhat from the other two. While the other problems have a significant amount of coupling between the control variables, the at is not the case here. This sparsity allows for Algorithm~2 to outperform both Gurobi and CPLEX (applied to the centralized problem). 

\subsection{Results for MIQP}\label{sec:miqp_sim}
If the comfort parameter $\gamma$ is larger than zero then Problem \eqref{eq:tcl} is a convex MIQP.\footnote{Convex in the same notion of convexity in MICP. That is, a problem where the continuous relaxation yields a convex quadratic program.} As in Section \ref{sec:milp_sim}, The results of Algorithm~2 for each room configuration are compared with those obtained from Bonmin, Gurobi, and CPLEX. The value of $\gamma$ was chosen to be one to allow for an equal weighting of comfort and cost.

Shown in Figure \ref{fig:params2} are the trajectories obtained for the three-room scenario with a temperature deviation penalization. In contrast to Figure \ref{fig:params}, a quadratic penalty term is used to model discomfort caused by deviations from the set temperature. Indeed, the solution with $\gamma=1$ yields a trajectory with a similar number of activations as when $\gamma=0$ but with temperature trajectories that stay much closer to the middle of the deadband. Theoretically, the quadratic term should make the problem more computationally difficult, but in some cases both Gurobi and CPLEX actually require less time. In contrast, Algorithm~2 requires many more iterations than the MILP formulation. Future work could seek to use some of the cutting plane methods and other heuristics used by Gurobi and CPLEX to alleviate this issue. Furthermore, quadratic lower bounding functions could significantly reduce the number of iterations until convergence.


\begin{table}[b!ht]	
	\setlength{\tabcolsep}{4pt}
	\renewcommand{\arraystretch}{1.3}
	\begin{center}
		\caption{Results obtained for Problem \eqref{eq:tcl} with a 3-room configuration.}
		\label{t:miqp3room}
		\begin{tabular}{|c|c|c|c|c|c|}
			\hline
			& Time steps:& 8 & 24  & 48  & 62 \\
			\hline
					& obj. 		&17.83	 &60.15	&104.07	&134.43	\\
			Alg.~2	& time(s) 	&0.24 & 0.49 & 1.36 & 3.17	\\
					& iter. 	&4	&3	&3	&3	\\
			\hline
			B-OA	&obj. 		&17.83	 &60.15	&104.07	&134.43	\\
					&time(s) 	&4.36 & 46.94 & 399.49 & 1,702.60	\\
			\hline
			B-BB	&obj. 		&17.83	 &60.15	&104.07	&134.43	\\
					&time(s) 	&9.09 & 70.40 & 502.76 & 1,152.90	\\
			\hline
			Gurobi	&obj. 		&17.83	 &60.15	&104.07	&134.43	\\
					&time(s) 	&0.44 & 0.50 & 0.66 & 0.83	\\
			\hline
			CPLEX   &obj. 		&17.83	 &60.15	&104.07	&134.43	\\
					&time(s) 	&0.19 & 0.27 & 0.25 & 0.54	\\
			\hline
		\end{tabular}
	\end{center}
\end{table}


\begin{table}[ht]	
	\setlength{\tabcolsep}{4pt}
	\renewcommand{\arraystretch}{1.3}
	\begin{center}
		\caption{Results obtained for Problem \eqref{eq:tcl} with a 4-room configuration.}
		\label{t:miqp4room}
		\begin{tabular}{|c|c|c|c|c|c|}
			\hline
			& Time steps:& 8 & 24  & 48  & 62 \\
			\hline
					& obj. 		&21.68	 &52.26	&99.68	&134.43	\\
			Alg.~2	& time(s) 	&0.37 & 0.65 & 8.82 & 31.07	\\
					& iter. 	&5	&3	&3	&4	\\
			\hline
			B-OA	&obj. 		&21.68	 &52.26	&99.68	&134.43	\\
					&time(s) 	&14.61 & 6.78 & 613.95 & 7,024.10\\
			\hline
			B-BB	&obj. 		&21.68	 &52.26	&99.68	&134.43	\\
					&time(s) 	&30.39 & 11.13 & 661.95 & 1,495.30	\\
			\hline
			Gurobi	&obj. 		&21.68	 &52.26	&99.68	&134.43	\\
					&time(s) 	&0.41 & 0.54 & 0.72 & 2.31	\\
			\hline
			CPLEX   &obj. 		&21.68	 &52.26	&99.68	&134.43	\\
					&time(s) 	&0.15 & 0.16 & 0.41 & 2.15	\\
			\hline
		\end{tabular}
	\end{center}
\end{table}


\begin{table}[ht]	
	\setlength{\tabcolsep}{4pt}
	\renewcommand{\arraystretch}{1.3}
	\begin{center}
		\caption{Results obtained for Problem \eqref{eq:tcl} with a linear 7 room configuration.}
		\label{t:miqp7room}
		\begin{tabular}{|c|c|c|c|c|c|}
			\hline
			& Time steps:& 8 & 24  & 48  & 62 \\
			\hline
					& obj. 		&39.50	 &112.08	&201.19	&267.04	\\
			Alg.~2	& time(s) 	&1.57 & 2.55 & 147.76 & 1022.38	\\
					& iter. 	&4	&3	&4	&3	\\
			\hline
			B-OA	&obj. 		&39.50	 &112.08	&N/A	&N/A	\\
					&time(s) 	&358.12 & 1,833.9 & N/A & N/A\\
			\hline
			B-BB	&obj. 		&39.50	 &112.08	&201.19	&267.04	\\
					&time(s) 	&593.17 & 955.77 & 6,095.6 & 15,582	\\
			\hline
			Gurobi	&obj. 		&39.50	 &112.08	&201.19	&267.04	\\
					&time(s) 	&0.95 & 0.72 & 10.83 & 65.59	\\
			\hline
			CPLEX   &obj. 		&39.50	 &112.08	&201.19	&267.04	\\
					&time(s) 	&0.93 & 0.77 & 6.53 & 17.50	\\
			\hline
		\end{tabular}
	\end{center}
\end{table}

\begin{figure}[th!]\label{fig:params2}
	\begin{center}
		\begin{tabular}{ccc}
			\includegraphics[width = 0.9\textwidth]{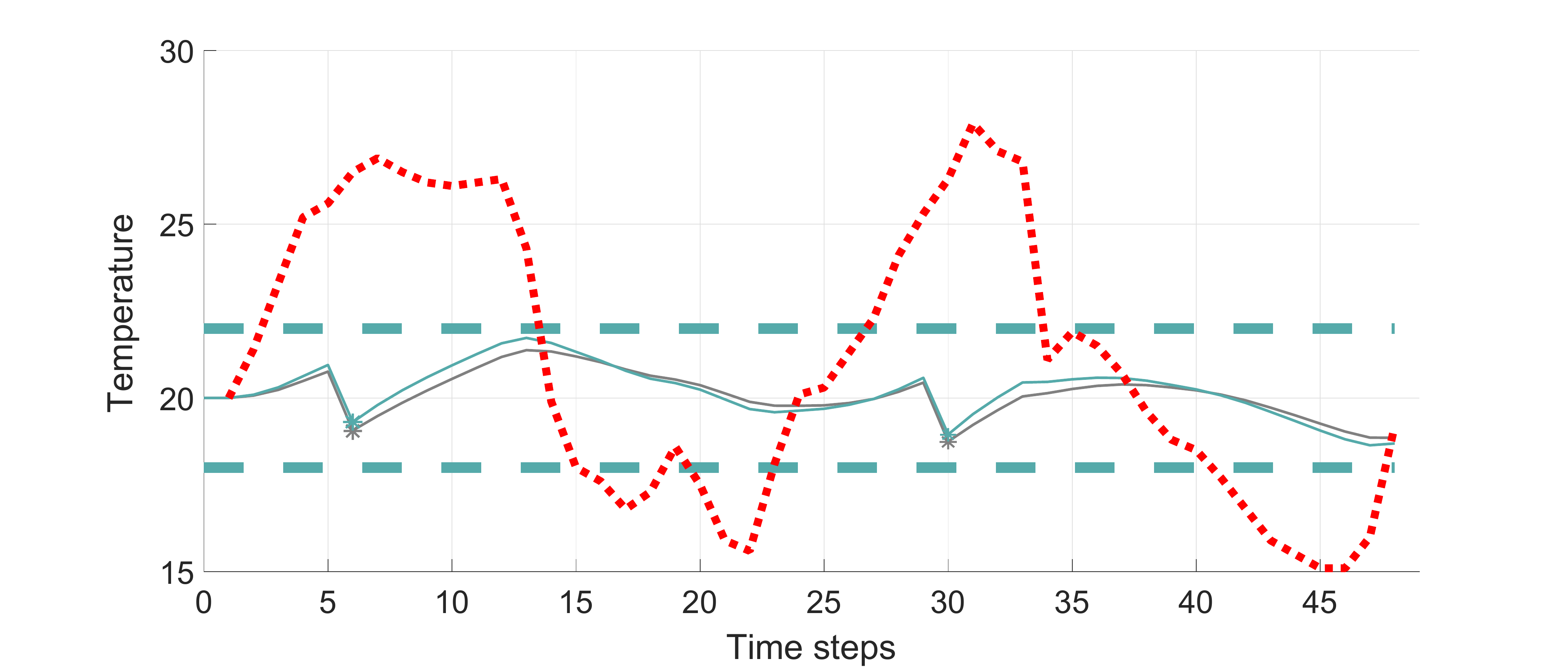}	
		\end{tabular}
	\end{center}
	\caption{Trajectories are defined as in Figure \ref{fig:params} for the three room scenario, except with the comfort parameter $\gamma=1$.}
\end{figure}

\subsection{Higher Order Convex Problems}
One of the advantages of the proposed algorithm is that it is applicable to a relatively large class of problems (namely, MICPs). While Section \ref{sec:miqp_sim} shows favourable results for both Gurobi and CPLEX, if the problem were adjusted slightly such that it were no longer an MIQP then these solvers would no longer be applicable. For example, if the objective function of Problem \eqref{eq:tcl} became
\begin{equation*}
\min\limits_{T(\cdot),u(\cdot)} \sum\limits_{t=0}^{H-1} c(t)u(t) + \gamma (T_i(t) - T_{ref}(t))^4,
\end{equation*}
\noindent then this would still be solvable via PaDOA, but not Gurobi or CPLEX. However, Bonmin can still be applied.\footnote{It should be noted that all results seen in this section for Algorithm~2 use Bonmin to solve the MICP subproblems.} The results for a variety of such problem configurations are shown below in Figure \ref{t:minlp}. Therein it can be observed that Algorithm~2 returns the same, global solution as Bonmin\footnote{With the Branch and Bound sub-algorithm.}, and does so in less time. The runtime difference is particularly striking for the 7 room scenarios as these contain the most variables and have the greatest potential for parallelization.

\begin{table}[ht]	
	\setlength{\tabcolsep}{4pt}
	\renewcommand{\arraystretch}{1.3}
	\begin{center}
		\caption{Results obtained for Problem \eqref{eq:tcl}, but with a $4^{th}$ order objective function.}
		\label{t:minlp}
		\begin{tabular}{|c|c|c|c|c|c|c|c|c|}
			\hline
			\multicolumn{2}{|c|}{}& \multicolumn{3}{c|}{Alg.~2}  & \multicolumn{2}{c|}{B-B\&B} & \multicolumn{2}{c|}{B-OA} \\
			\hline
			Rooms	&Time steps	& obj. & time(s)  & iter.  & obj. & time(s)  & obj. & time(s)  \\
			\hline
			3&  8& 16.72 & 4.20 & 5 & 16.72 & 7.51 & N/A & N/A \\
			3& 24& 76.86 & 19.14 & 4 & 76.86 & 308.64 & N/A & N/A  \\
			3& 48& 115.62 & 95.56 & 3 & 115.62 & 547.98 & N/A & N/A \\
			3& 62& 141.44 & 405.19 & 6 & 141.44 & 779.41 & N/A & N/A \\
			4&  8& 21.73 & 4.11 & 6 & 21.73 & 16.93 & N/A & N/A \\
			4& 24& 44.94 & 6.47 & 3 & 44.94 & 9.93 & N/A & N/A  \\
			4& 48& 86.72 & 151.58 & 5 & 86.72 & 477.60 & N/A & N/A \\
			4& 62& 120.71 & 175.78 & 6 & 120.71 & 781.46 & N/A & N/A \\
			7&  8& 38.41 & 8.87 & 7 & 38.41 & 374.38 & N/A & N/A \\
			7& 24& 122.09 & 34.86 & 3 & 122.09 & 1291.29 & N/A & N/A  \\
			7& 48& 202.54 & 1,631.6 & 4 & 202.54 & 2,830.93 & N/A & N/A \\
			7& 62& 263.91 & 3614.54 & 4 & 263.91 & 9453.89 & N/A & N/A \\
			\hline
		\end{tabular}
	\end{center}
\end{table}

\begin{table}[ht]	
	\setlength{\tabcolsep}{4pt}
	\renewcommand{\arraystretch}{1.3}
	\begin{center}
		\caption{Runtime breakdown of Algorithm~2 applied to the $2^{nd}$-order version of Problem \ref{eq:tcl} with 7 room TCL problem with 48 time steps. }
		\label{t:runtime_breakdown}
		\begin{tabular}{|c|c|c|c|c|c|}
			\hline
			& Iter. 1 & Iter. 2  & Iter. 3  & Iter. 4   \\
			\hline
			MINLP time (s)	    & 0.27 & 0.20 &	0.20 & 0.20  \\
			MILP time(s)	    & 2.20 & 86.89 & 82.62 & 87.62  \\
			Hyperplane time (s)	& 0.001 & 0.005 &	0.002 & 0.002  \\
			\hline
		\end{tabular}
	\end{center}
\end{table}

\begin{figure}[th!]\label{fig:conv}
	\begin{center}
		\begin{tabular}{ccc}
			\includegraphics[width = 0.9\textwidth]{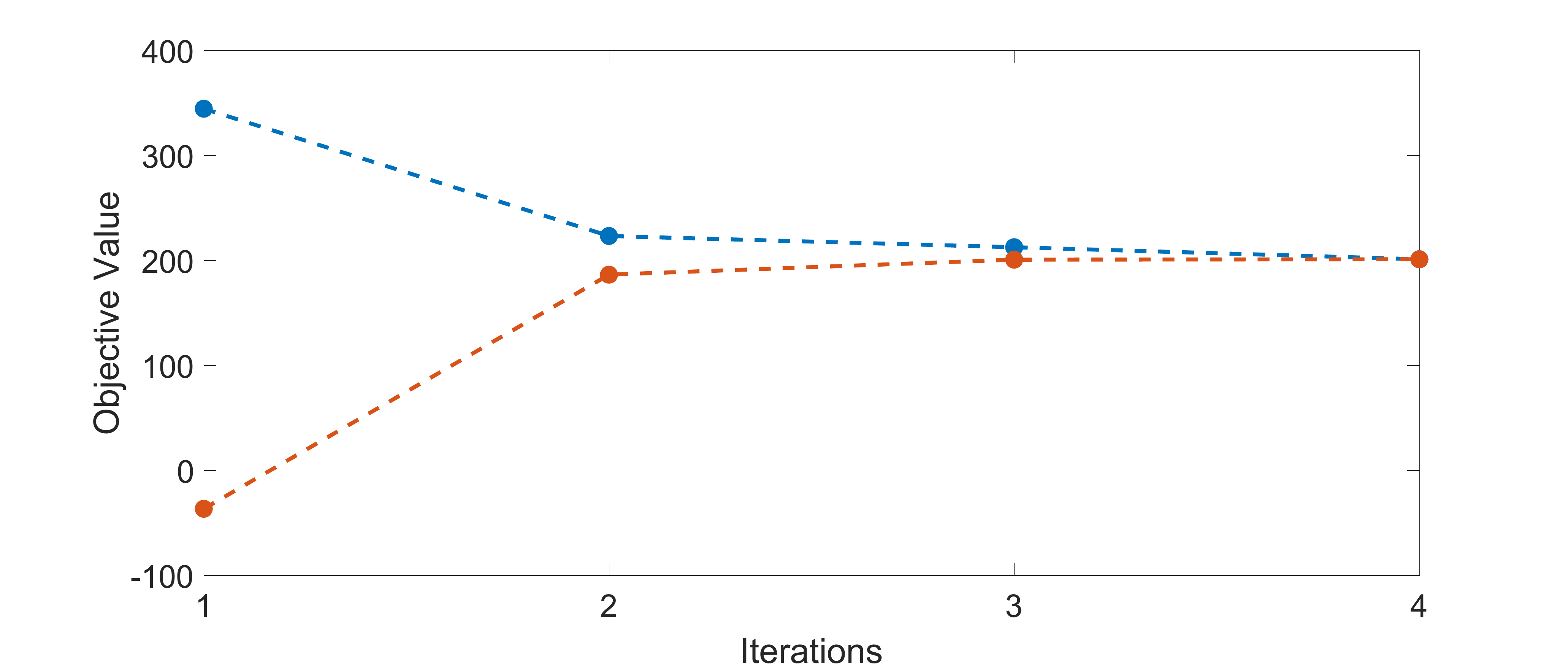}	
		\end{tabular}
	\end{center}
	\caption{Progression of the upper and lower bounds during each iteration while solving the $2^{nd}$-order version of Problem \ref{eq:tcl} with 7 rooms and 48 time steps. The blue line depicts the progression of the upper bound and the red shows that of the lower bound.}
\end{figure}

\subsection{Outlook}
The implementation of Algorithm~2 used in Section \ref{sec:sim} was protypical and could be improved in a number of ways. Interestingly, as Table \ref{t:runtime_breakdown} shows, the majority of the time spent by the algorithm was typically on the coupling problems. However, more time was required for the solution of the MICP subproblems in the higher order version of the problem. An example of which is shown in Table \ref{t:runtime_breakdown2}, along with the relevant convergence plot.

The reason why the linear approximations required so much time to be solved likely lies in the heuristics and presolving processes used by Gurobi. This is despite the fact that Gurobi was used to solve both the linear approximations and the full nonlinear program.

\begin{table}[ht]	
	\setlength{\tabcolsep}{4pt}
	\renewcommand{\arraystretch}{1.3}
	\begin{center}
		\caption{Runtime breakdown of Algorithm~2 applied to the $4^{th}$-order version of Problem \ref{eq:tcl} with 4 rooms and 8 time steps.}
		\label{t:runtime_breakdown2}
		\begin{tabular}{|c|c|c|c|c|c|c|c|}
			\hline
			& Iter. 1 & Iter. 2  & Iter. 3  & Iter. 4 & Iter. 5 & Iter. 6  \\
			\hline
			MINLP time (s)	    & 0.73 & 0.75 &	0.67 & 0.65 & 0.65 & 0.66\\
			MILP time(s)	    & 0.03 & 0.04 & 0.04 & 0.04 & 0.04 & 0.04 \\
			Hyperplane time (s)	& 0.004 & 0.006 &	0.002& 0.002 & 0.002&0.002 \\
			\hline
		\end{tabular}
	\end{center}
\end{table}

\begin{figure}[th!]\label{fig:conv2}
	\begin{center}
		\begin{tabular}{ccc}
			\includegraphics[width = 0.9\textwidth]{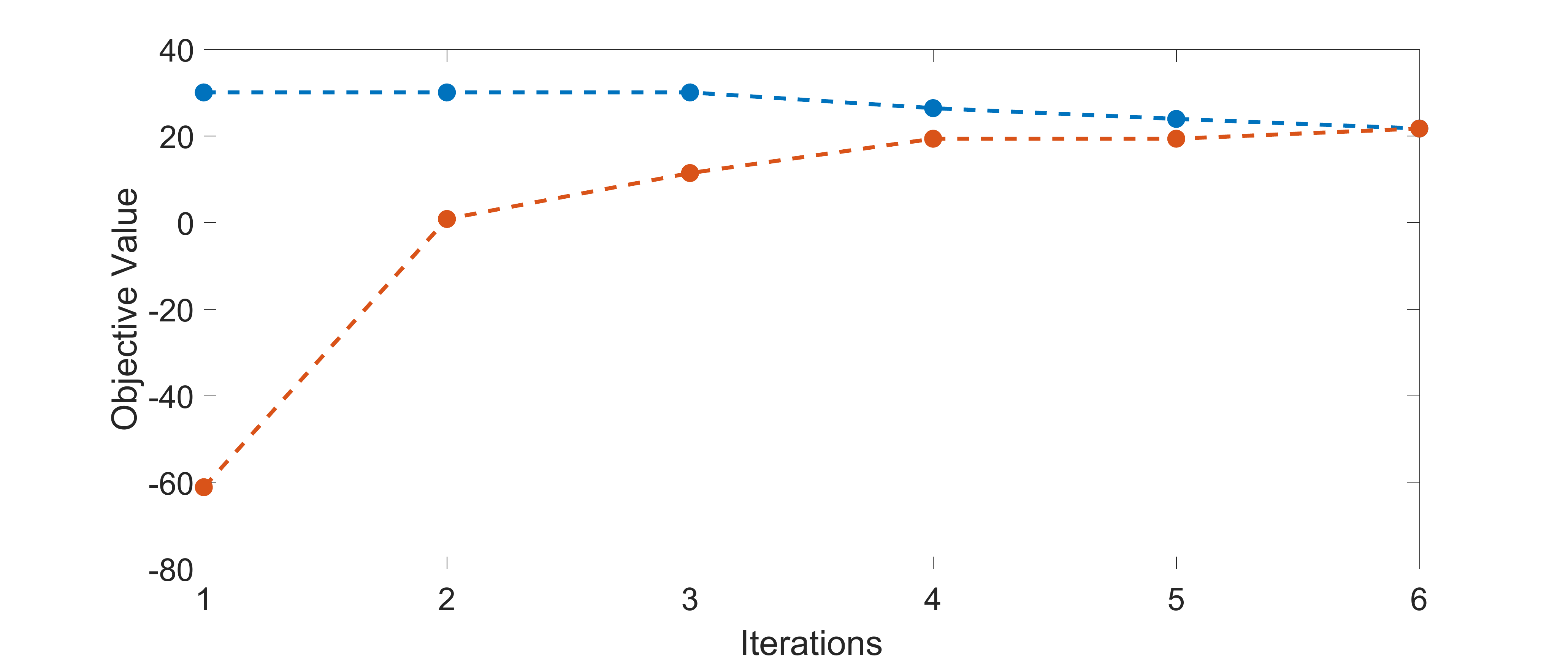}	
		\end{tabular}
	\end{center}
	\caption{Progression of the upper and lower bounds during each iteration while solving the $4^{th}$-order version of Problem \ref{eq:tcl} with 4 rooms and 8 time steps. The blue line depicts the progression of the upper bound and the red shows that of the lower bound.}
\end{figure}


\section{Conclusions}
\label{sec::conclusions}
This paper has introduced the partially distributed outer approximation method PaDOA (Algorithm~2) for finding $\epsilon$-suboptimal points of the structured MICP~\eqref{eq::minlp}. PaDOA proceeds by alternating between solving partially decoupled MICPs that comprise $n$ integer variables and large scale MILPs with $nN$ variables. Finite termination conditions for PaDOA have been established in Theorem~\ref{thm::PaDOA}. Moreover, we have discussed the major theoretical and practical advantages of PaDOA compared to exising extended formulation based OA solvers. In particular, Theorem~\ref{thm::termination} states that PaDOA terminates after the first iteration, if it is initialized at a global minimizer---an important property that is neither shared by existing OA nor by existing branch-and-bound based methods for MICP.

In Section \ref{sec:sim}, first, second and fourth order mixed integer problems were used to demonstrate the practical performance of PaDOA compared to other state of the art solvers by application to a scheduling problem of thermostatically controlled loads. While the solution and runtime are competitive for each of the case studies considered, the best performance was observed for problems with sparse Hessians and coupling constraints. Furthermore, it was observed that PaDOA was able to return a solution in several cases where the centralized approach could not due to memory constraints.


Future work will investigate the use of piecewise linear underapproximations and/or cutting planes in the MILP step to extend PaDOA to non-convex MINLPs. Furthermore, Step \ref{step::MILP} requires full constraint information in order to return a feasible solution. This restricts the applicability in terms of fully distributed settings and future work will focus on sidestepping this restriction.


\bibliographystyle{plain}
\bibliography{BAM_bib}   

\appendix

\section{Proofs}

\subsection{Proof of Proposition~\ref{prop::AuxOptDual}}
\label{app::AuxOptDual}

If there is no $x \in X$ with $Ax = b$, both sides of~\eqref{eq::AuxOptDual} are equal to infinity and the statement of the proposition holds in the extended value sense. Thus, we may assume that the constraints in~\eqref{eq::AuxOptDual} are feasible. Consequently, Assumption~\ref{ass::convexity} implies that~\eqref{eq::AuxOpt} is a convex optimization problem with compact and non-empty feasible set. Moreover, since $X$ is a polytope, all constraints in~\eqref{eq::AuxOpt} are linear. It is well-known~\cite{Bertsekas1999,Rockafellar1970} that strong duality holds under these conditions.
\QED

\subsection{Proof of Lemma~\ref{lemma::lowerBound}}
\label{app::lowerBound}
If there is no $x \in X$ with $A x = b$, both sides of~\eqref{eq::TightLowerBound} are equal to infinity and the statement of the lemma holds in the extended value sense. Thus, we may assume that the equation $Ax = b$ has a solution in $X$. Next, because we have $z \in \Xi$, our particular construction of $\Phi$ implies that
\begin{eqnarray}
\Phi(x,z,\Xi) &\geq& f^\star(z) + \left[ \lambda^\star(z) \right]^\tr ( x - x^\star(z) ) + \underbrace{\left[ \mu^\star(z) \right]^\tr ( z - z )}_{=0}
\end{eqnarray}
Thus, we have
\begin{eqnarray}
\min_{x \in X, Ax = b} \; \Phi(x,z,\Xi) &=& \min_{y \in X, Ay = b} \;  f^\star(z) + \left[ \lambda^\star(z) \right]^\tr ( y - x^\star(z) ) \notag \\[0.16cm]
&=& \min_{x,y} \; f(x,z) + \left[ \lambda^\star(z) \right]^\tr ( y - x )  \quad \mathrm{s.t.} \quad \left\{
\begin{array}{l}
A y = b \\
y \in X
\end{array}
\right. \notag \\[0.25cm]
&=& \max_{\lambda} \; \min_{x,y} \; f(x,z) + \lambda^\tr (y-x) \quad \mathrm{s.t.} \quad \left\{
\begin{array}{l}
A y = b \\
y \in X
\end{array}
\right. \; . \notag
\end{eqnarray}
Since Assumption~\ref{ass::convexity} holds, we may substitute~\eqref{eq::AuxOptDual} (see Proposition~\ref{prop::AuxOptDual}), which yields the equation
\[
\forall z \in \Xi, \qquad  \min_{x \in X, Ax = b} \; \Phi(x,z,\Xi) = f^\star(z) \; .
\]
\QED

\subsection{Proof of Theorem~\ref{thm::termination1}}
\label{app::termination1}

Notice that if the equation $Ax = b$ has no solution in $X$, this will be detected immediately by Step~2 of Algorithm~1, which causes termination. Thus, we may assume that all optimization problems are feasible. Now, the main idea of the proof is to show that the cardinality of the set $\Pi$ is strictly increasing in every iteration, if the algorithm does not terminate. For this aim, we first notice that any solution $(x^+,y^+,z^+)$ of the MILP~\eqref{eq::MILP1} satisfies the equation
\begin{align}
\label{eq::AUX1}
\Phi(x^+,z^+,\Pi) = \sum_{i=1}^N y_i^+
\end{align}
by construction. Moreover, because we have $A x^+ = b$, the inequality
\begin{align}
\label{eq::AUX11}
\min_{x,Ax = b} \; \Phi(x,z^+,\Pi) \leq \Phi(x^+,z^+,\Pi)
\end{align}
holds. If we further assume that the termination criterion is not satisfied, we must have
\begin{align}
\label{eq::AUX2}
\sum_{i=1}^N y_i^+ < U - \epsilon
\end{align}
Thus, if we had $z^+ \in \Xi$, then the result of Lemma~\ref{lemma::lowerBound} would imply that
\begin{align}
\label{eq::AUX22}
f^\star(z^+) \overset{\eqref{eq::TightLowerBound}}{=} \min_{x,Ax = b} \; \Phi(x,z^+,\Pi)
\end{align}
as well as $U \leq f(z^+)$, since $z^+$ has already been added to the collection $\Pi$. By substituting all the above relations we would then find that
\[
f^\star(z^+) \, \overset{\eqref{eq::AUX22},\eqref{eq::AUX11}}{\leq} \, \Phi(x^+,z^+,\Pi) \, \overset{\eqref{eq::AUX1},\eqref{eq::AUX2}}{<} \, U - \epsilon \, \leq \, f(z^+) - \epsilon \; ,
\]
which is a contraction. Thus, either our assumption that the algorithm does not terminate or our assumption $z^+ \in \Xi$ must be wrong. In other words, if the algorithm does not terminate in the current step, then the cardinality of the set $\Pi$ increases by $1$ in the next step, because $z^+$ is added to the collection $\Pi \subseteq Z$. But this is only possible for a finite number of steps, because the set $Z$ contains only a finite number of points. Thus, Algorithm~1 must terminate after a finite number of iterations.
\QED

\subsection{Proof of Theorem~\ref{thm::termination}}
\label{app::termination}

Let $V^\star = \sum_{i=1}^N f_i(x_i^\star,z_i^\star)$ denote the optimal value of~\eqref{eq::minlp}. Because we assume that such an optimal solution exists while Assumption~\ref{ass::convexity} is satisfied, the partially decoupled optimization problems are all feasible and return piecewise affine lower bounds that satisfy the termination condition~\eqref{eq::LowerLevelTermination} with $V_k(z^\star) = V^\star$, i.e., we have
\begin{eqnarray}
V^\star - \epsL \; \leq \min_{x \in X,\zeta \in Z_k} \Theta_k^\star(x,\zeta) \quad \mathrm{s.t.} \quad A x = b 
\end{eqnarray}
for all $k \in \{ 1, \ldots, N \}$. Because the function $\Theta$ is by construction an upper bound on $\Theta_k$ (for any $k$), we further have
\[
\min_{x \in X,\zeta \in Z_k,Ax = b} \Theta_k^\star(x,\zeta) \leq  \min_{x \in X,z \in Z,Ax = b} \Theta(x,z) = \Theta(x^+,z^+) \; ,
\]
where $(x^+,z^+)$ denotes the solution of the master MILP~\eqref{eq::MILP}. By substituting the above inequalities we find that
\[
V^\star - \epsL \leq \Theta(x^+,z^+) \; .
\]
Because we assume that $\epsL \leq \epsilon$, this implies that
\[
U - \Theta(x^+,z^+) = V^\star - \Theta(x^+,z^+) \leq \epsL \leq \epsilon \; .
\]
Thus, the termination condition is satisfied and Algorithm~2 terminates after the first step.
\QED

\subsection{Proof of Theorem~\ref{thm::PaDOA}}
\label{app::PaDOA}

We may assume that the coupled equality constraint is feasible, as infeasibility would be detected immediately in Step~1 of Algorithm~2. Similar to the proof of Theorem~\ref{thm::termination1}, we need to keep track of the integer solutions of the master MILPs. For this aim, we introduce the following ``artificial'' additional step:
\begin{quote}
Step $3'$): After solving~\eqref{eq::MILP}, update $\tilde \Pi = \tilde \Pi \cup \{ z^+ \}$.
\end{quote}
If the set $\tilde \Pi$ is initialized with the empty set and if Step $3'$ is inserted in Algorithm~2 immediately after Step~3, the iterates of this algorithm remain unaffected. The main idea of the proof is now to show that the cardinality of the set $\tilde \Pi$ increases in every iteration of Algorithm~1 under the assumption that the termination criterion is not satisfied. Let us assume that the solution $z^+$ satisfies $z^+ \in \tilde \Pi$ (before $\tilde \Pi$ is updated in Step $3'$). Then we have
\[
U \leq V_k(z^+) \overset{\eqref{eq::LowerLevelTermination}}{\leq} \epsL + \min_{x \in X,\zeta \in Z_k, Ax = b} \Theta_k^\star(x,\zeta) \; ,
\]
Because $\Theta$ is an upper bound on $\Theta_k$, this implies that we also have
\[
\min_{x \in X,\zeta \in Z_k, Ax = b} \Theta_k^\star(x,\zeta) \leq \min_{x \in X,z \in Z, Ax = b} \Theta^\star(x,\zeta) = \Phi(x^+,z^+) \; ,
\]
which yields $U - \Phi(x^+,z^+) \leq \epsL \leq \epsilon$. Thus, either the termination criterion is satisfied  or we have $z^+ \notin \tilde \Pi$. In the latter case, the cardinality of $\Pi$ increases by $1$ in the current iteration. As this is only possible for a finite number of steps, Algorithm~2 must terminate.
\QED

\end{document}